\documentclass[oneside,reqno,english]{amsart}
\usepackage[T1]{fontenc}
\usepackage[utf8]{inputenc}
\usepackage{xcolor}
\usepackage{babel}
\usepackage{prettyref}
\usepackage{amstext}
\usepackage{amsthm}
\usepackage{amssymb}
\usepackage[pdfusetitle,
 bookmarks=true,bookmarksnumbered=false,bookmarksopen=false,
 breaklinks=false,pdfborder={0 0 0},pdfborderstyle={},backref=false,colorlinks=false]
 {hyperref}
\hypersetup{
 colorlinks=true,citecolor=blue,linkcolor=blue,linktocpage=true}

\makeatletter
\numberwithin{equation}{section}
\numberwithin{figure}{section}

\usepackage{prettyref}

\newrefformat{cor}{Corollary~\ref{#1}}
\newrefformat{subsec}{Section~\ref{#1}}
\newrefformat{lem}{Lemma~\ref{#1}}
\newrefformat{thm}{Theorem~\ref{#1}}
\newrefformat{sec}{Section~\ref{#1}}
\newrefformat{chap}{Chapter~\ref{#1}}
\newrefformat{prop}{Proposition~\ref{#1}}
\newrefformat{exa}{Example~\ref{#1}}
\newrefformat{tab}{Table~\ref{#1}}
\newrefformat{rem}{Remark~\ref{#1}}
\newrefformat{def}{Definition~\ref{#1}}
\newrefformat{fig}{Figure~\ref{#1}}
\newrefformat{claim}{Claim~\ref{#1}}
\newrefformat{assu}{Assumption~\ref{#1}}

\makeatother

\theoremstyle{plain}
\newtheorem{thm}{\protect\theoremname}[section]
\theoremstyle{definition}
\newtheorem{defn}[thm]{\protect\definitionname}
\theoremstyle{plain}
\newtheorem{lem}[thm]{\protect\lemmaname}
\newtheorem{cor}[thm]{\protect\corollaryname}
\theoremstyle{remark}
\newtheorem*{rem*}{\protect\remarkname}
\newtheorem{rem}[thm]{\protect\remarkname}
\theoremstyle{plain}
\newtheorem{prop}[thm]{\protect\propositionname}
\theoremstyle{definition}
\newtheorem{example}[thm]{\protect\examplename}
\providecommand{\corollaryname}{Corollary}
\providecommand{\definitionname}{Definition}
\providecommand{\examplename}{Example}
\providecommand{\lemmaname}{Lemma}
\providecommand{\propositionname}{Proposition}
\providecommand{\remarkname}{Remark}
\providecommand{\theoremname}{Theorem}

\begin{document}
\title[]{Wavelet-Packet Content for Positive Operators}
\begin{abstract}
We study positive operator decompositions associated with rooted trees
of orthogonal projections. In this sense, the refinement tree induces
an ``MRA in $\frac{}{}B(H)_{+}$''. To each node we assign a positive
content operator, and these contents split along the tree and yield
a positive decomposition at each fixed depth. The resulting decomposition
gives a multiresolution description of positive operators adapted
to the tree. In the trace class setting, the scalar contents determine
a canonical boundary measure on the path space, and for each vector
the corresponding quadratic data admit a nonnegative integrable density
with respect to that measure. At fixed depth, we study greedy extraction
rules based on trace and Hilbert-Schmidt norm. The trace rule gives
a sharp geometric decay estimate for the trace of the positive remainder.
In the Hilbert-Schmidt setting, a depth dependent coherence parameter
measures departure from block diagonal form and yields geometric decay
bounds. We also study adaptive partitions up to a terminal depth.
In that setting, the change in total squared content under local refinement
is determined by off-diagonal interaction among the child contents.
This leads to an additive refinement calculus for adaptive decompositions
and quantitative bounds for adaptive greedy extraction.
\end{abstract}

\author{Myung-Sin Song}
\address{(Myung-Sin Song) Department of Mathematics, Southern Illinois University
Edwardsville, Edwardsville, IL 62026, U.S.A.}
\email{msong@siue.edu}
\urladdr{http://www.siue.edu/\textasciitilde msong/}
\author{James Tian}
\address{(James Tian) Mathematical Reviews, 535 W. William St, Suite 210, Ann
Arbor, MI 48103, USA}
\email{james.ftian@gmail.com}
\maketitle

\section{Introduction}\label{sec:1}

Wavelet packets provide a familiar model for a refinement tree of
orthogonal subspaces. In the present setting, the term ``wavelet
packet'' serves mainly to indicate that refinement picture, and the
analysis depends only on the associated tree of orthogonal projections
and the refinement relations among them. 

Fix a family of subspaces $\left\{ W_{w}:w\in\mathcal{T}\right\} $
with the associated orthogonal projections $\left\{ P_{w}\right\} _{w\in\mathcal{T}}$
satisfying the usual tree relations. Given a positive operator $R\in B\left(H\right)_{+}$,
the simplest node-wise quantities are the compressions $P_{w}RP_{w}$
or the scalar weights $\mathrm{tr}\left(P_{w}R\right)$ and $\left\langle x,P_{w}Rx\right\rangle $.
For the present problem, the more useful object is the content operator
\[
C_{w}\left(R\right)=R^{1/2}P_{w}R^{1/2}\in B\left(H\right)_{+}.
\]
These content operators stay in the positive cone, telescope along
the tree, and at each fixed depth $n$ give a positive decomposition
\[
R=\sum_{\left|w\right|=n}C_{w}\left(R\right).
\]
Thus every partial sum is positive, and so is the remainder.

The refinement tree acts strictly on the positive cone, giving a multiresolution
type decomposition with an operator level notion of scale and localization.
We consider two related questions.
\begin{enumerate}
\item Fixed depth selection. At a chosen depth, one compares blocks by trace
or by Hilbert-Schmidt (HS) size and asks how fast the positive remainder
decreases under repeated removal of a single block.
\item Adaptive selection. One allows the partition itself to vary along
the tree and asks how the HS error depends on the nodes that are refined. 
\end{enumerate}
In the trace class setting, the content operators determine a boundary
measure on the path space of the tree. In \prettyref{sec:3} we show
that there is a unique finite Borel measure $\mu_{R}$ on $\partial\mathcal{T}$
such that 
\[
\mu_{R}\left(\left[w\right]\right)=\mathrm{tr}\left(C_{w}\left(R\right)\right)
\]
for every node $w$, and that for each $x\in H$ there is a nonnegative
density $e_{x}\in L^{1}\left(\partial\mathcal{T},\mu_{R}\right)$
representing the quadratic data $\left\langle x,C_{w}\left(R\right)x\right\rangle $
on cylinder sets. This boundary representation is useful for the fixed
depth trace rule, where one repeatedly removes the heaviest cylinder
at a given level.

In \prettyref{sec:4}, the trace-based rule removes a depth $n$ block
of maximal trace and gives a geometric decay estimate for the trace
of the positive remainder. The HS analogue is treated in \prettyref{sec:5}.
There the interaction among different blocks enters essentially, and
a depth $n$ coherence parameter $\Gamma_{n}\left(A\right)$ measures
how far a positive HS operator is from being block diagonal relative
to the depth $n$ partition. The resulting estimates give geometric
decay in HS norm.

In \prettyref{sec:6}, the partition is allowed to vary up to a terminal
depth $N$. For an active partition $\Lambda$, the change in total
squared HS content under local refinement is an explicit nonnegative
defect determined by the off-diagonal interaction among the child
contents. These defects give an additive refinement calculus along
adaptive chains of partitions and lead to adaptive trace and Hilbert-Schmidt
greedy bounds.

The fixed depth restriction amounts to choosing a resolution. Increasing
$n$ gives finer localization but also increases the number $N_{n}$
of available blocks, which enters the bounds in \prettyref{thm:4-1}
and \prettyref{thm:5-3}. The adaptive results in \prettyref{sec:6}
replace the fixed level by a variable tree cut.

\subsection*{Literature context}

Wavelet packets, best basis selection, and packet based compression
have a substantial literature. Classical packet constructions and
optimal decomposition methods appear in \cite{MR1155278,MR1282933,Wickerhauser1994AdaptedWA,Tuthill1995BookRA},
with later developments for packet libraries, graph based constructions,
and imaging applications in \cite{MR2747055,MR4248654,MR4580931}.
There are also multilevel systems and related orthogonal refinement
structures close to the present tree setting; see, for example, \cite{MR2001190,doi:10.1049/ip-vis:20030873,MR4619155,MR3103223,MR4008697,Dutkay2013GeneralizedWB}. 

Related positive operator decompositions based on residual weighted
updates, together with associated path space constructions, appear
in earlier work. In \cite{MR5079224}, the rank-one setting leads
to monotone residual decompositions and exhaustion criteria, and \cite{MR5055731}
develops a contraction valued residual tree whose transition probabilities
and boundary measures are generated from the evolving residuals.

The construction here uses the same positive splitting idea in a more
rigid setting. The rooted tree of orthogonal projections is fixed
in advance, and the packet content construction from \prettyref{def:2-1},
together with the refinement identity \prettyref{eq:2-3}, gives compatible
positive decompositions at every depth. In particular, the tree structure
is fixed independently of the residual evolution, rather than being
generated branchwise from evolving residual operators. This fixed
tree setting gives new sharp fixed depth trace decay estimates and
an exact additive refinement calculus for active partitions.

The nonlinear map $R\mapsto R^{1/2}PR^{1/2}$ is related to several
operator-theoretic themes, including shorting, Schur complement type
constructions, and related compression procedures \cite{MR287970,MR852902,MR2234254,MR2284176}.
There are also connections with classical work on operator ranges
and the geometry of pairs of projections \cite{MR293441,MR531986,MR2580440}.
Here the same projection tree is used for both the fixed depth and
adaptive decompositions.

There is also an analogy with block selection procedures in projection
based iterative methods, such as Kaczmarz type algorithms \cite{MR2311862,MR2500924,MR3134343,jeong2025infinitedimensionaloperatorblockkaczmarzalgorithms}.
The fixed depth extraction rules in Sections \ref{sec:4} and \ref{sec:5}
resemble greedy block selection schemes. The difference here is that
both the extracted pieces and the remainders remain in $B\left(H\right)_{+}$,
and the remainder chain is monotone in the Loewner order, giving a
canonical limiting residual operator without additional geometric
assumptions.

\section{Packet content operators}\label{sec:2}

Let $H$ be a Hilbert space. Let $\mathcal{T}$ be a rooted finite
branching tree, and let 
\[
\left\{ W_{w}\subset H:w\in\mathcal{T}\right\} 
\]
be a family of closed subspaces with orthogonal projections $\left\{ P_{w}\right\} _{w\in\mathcal{T}}$,
subject to the following: 

For each $n\ge0$, the subspaces $\left\{ W_{w}:\left|w\right|=n\right\} $
form an orthogonal decomposition of $H$, where
\begin{equation}
H=\bigoplus_{\left|w\right|=n}W_{w},\qquad I=\sum_{\left|w\right|=n}P_{w}.\label{eq:b-1}
\end{equation}

For each node $w\in\mathcal{T}$, there is a finite set $\mathrm{ch}\left(w\right)\subset\mathcal{T}$
with 
\begin{equation}
W_{w}=\bigoplus_{v\in\mathrm{ch}\left(w\right)}W_{v},\qquad P_{w}=\sum_{v\in\mathrm{ch}\left(w\right)}P_{v},\label{eq:b-2}
\end{equation}
and 
\[
P_{v}P_{v'}=0\quad(v\ne v',\,v,v'\in\mathrm{ch}\left(w\right)).
\]

\begin{defn}
\label{def:2-1}For each $w\in\mathcal{T}$ and each $R\in B\left(H\right)_{+}$,
the packet content operator of $R$ at $w$ is 
\[
C_{w}\left(R\right):=R^{1/2}P_{w}R^{1/2}\in B\left(H\right)_{+}.
\]
\end{defn}

At each fixed depth $n$, the relation \prettyref{eq:b-1} gives the
positive decomposition
\[
R=\sum_{\left|w\right|=n}R^{1/2}P_{w}R^{1/2}=\sum_{\left|w\right|=n}C_{w}\left(R\right),
\]
and for every node $w\in\mathcal{T}$, the splitting \eqref{eq:b-2}
gives the refinement identity
\begin{equation}
C_{w}\left(R\right)=\sum_{v\in\mathrm{ch}\left(w\right)}C_{v}\left(R\right).\label{eq:2-3}
\end{equation}
Thus the family $\left\{ C_{w}\left(R\right)\right\} _{w\in\mathcal{T}}$
is an operator-valued partition of $R$ adapted to the tree.

For each $x\in H$ and each node $w$, we have 
\[
\left\langle x,C_{w}\left(R\right)x\right\rangle =\left\langle R^{1/2}x,P_{w}R^{1/2}x\right\rangle =\left\Vert P_{w}R^{1/2}x\right\Vert ^{2},
\]
so the quadratic form of $C_{w}\left(R\right)$ at $x$ can be interpreted
as the packet energy of the vector $R^{1/2}x$ in the subspace $W_{w}$. 

\subsection{Sequential content extraction}

We now introduce a sequential extraction process on $B\left(H\right)_{+}$.
Fix an arbitrary sequence of nodes $\left(w_{k}\right)_{k\ge1}$ in
$\mathcal{T}$. Starting from $R^{\left(0\right)}=R\in B\left(H\right)_{+}$,
we define inductively for $k\ge1$ 
\[
D_{k}=C_{w_{k}}(R^{\left(k-1\right)})=R^{\left(k-1\right)1/2}P_{w_{k}}R^{\left(k-1\right)1/2},
\]
\[
R^{\left(k\right)}=R^{\left(k-1\right)}-D_{k}=R^{\left(k-1\right)1/2}\left(I-P_{w_{k}}\right)R^{\left(k-1\right)1/2}.
\]
Each $D_{k}$ is a positive operator extracted at step $k$, and each
$R^{\left(k\right)}$ is also positive. The construction is nonlinear:
the operator $D_{k}$ depends on the entire history $\left(w_{1},\dots,w_{k}\right)$
through  $R^{\left(k-1\right)}$, and the order in which nodes are
chosen matters in general.

In particular, we have
\[
0\le R^{\left(k\right)}\le R^{\left(k-1\right)}\le\cdots\le R^{\left(0\right)}=R
\]
for all $k\ge1$. 
\begin{lem}
\label{lem:2-2}There exists $R^{\left(\infty\right)}\in B\left(H\right)_{+}$
such that $R^{\left(k\right)}\to R^{\left(\infty\right)}$ in the
strong operator topology and $0\le R^{\left(\infty\right)}\le R^{\left(k\right)}$
for all $k$. Moreover, for every $N\ge1$ one has 
\begin{equation}
R=R^{\left(N\right)}+\sum^{N}_{k=1}D_{k},\label{eq:2-4-1}
\end{equation}
and the series $\sum_{k\ge1}D_{k}$ converges strongly to $R-R^{\left(\infty\right)}$. 
\end{lem}

\begin{proof}
For each $x\in H$ the sequence $\left\langle x,R^{\left(k\right)}x\right\rangle $
is nonincreasing and bounded below by $0$, hence convergent. Thus
the quadratic forms associated to $R^{\left(k\right)}$ decrease pointwise
to a bounded nonnegative form $q$ on $H$. By the standard monotone
convergence theorem for positive operators on a Hilbert space, there
is a unique $R^{\left(\infty\right)}\in B\left(H\right)_{+}$ such
that $R^{\left(k\right)}\to R^{\left(\infty\right)}$ strongly and
$0\le R^{\left(\infty\right)}\le R^{\left(k\right)}$ for all $k$.
The identity \prettyref{eq:2-4-1} follows by telescoping. 
\end{proof}

\begin{defn}
We call $\left(D_{k}\right)_{k\ge1}$ a content extraction sequence
for $R$ along the node sequence $\left(w_{k}\right)_{k\ge1}$, and
$R^{\left(\infty\right)}$ the limit associated to that sequence.
For each $N$ the partial sum $\sum^{N}_{k=1}D_{k}$ is an $N$-term
approximation of $R$, with  error $R^{\left(N\right)}$.
\end{defn}

The remainder $R^{\left(\infty\right)}$ depends in general on the
chosen node sequence. It does vanish, however, under the following
condition.
\begin{thm}
\label{thm:2-4} Let $\Lambda\subset\mathcal{T}$ be a finite set
of nodes such that 
\[
S=\sum_{w\in\Lambda}C_{w}\left(S\right)
\]
for every $S\in B\left(H\right)_{+}$. Let $\left(w_{k}\right)_{k\ge1}$
be a sequence of nodes in $\mathcal{T}$, and let $\left(R^{\left(k\right)}\right)_{k\ge0}$
be the associated remainder sequence. Suppose that every node $w\in\Lambda$
occurs infinitely often in $\left(w_{k}\right)_{k\ge1}$. Then 
\[
R^{\left(\infty\right)}=0.
\]
\end{thm}

\begin{proof}
Fix $w\in\Lambda$, and choose a subsequence $\left(k_{j}\right)_{j\ge1}$
such that $w_{k_{j}}=w$ for all $j$. By \prettyref{lem:2-2}, we
get (strong convergence)
\[
R^{\left(k_{j}-1\right)}\xrightarrow{s}R^{\left(\infty\right)}\qquad\text{and}\qquad R^{\left(k_{j}\right)}\xrightarrow{s}R^{\left(\infty\right)}.
\]
Since $R^{\left(k_{j}\right)}=R^{\left(k_{j}-1\right)}-C_{w}(R^{\left(k_{j}-1\right)})$,
it follows that $C_{w}(R^{\left(k_{j}-1\right)})\xrightarrow{s}0$. 

Because the operators $R^{\left(k_{j}-1\right)}$ are positive, uniformly
bounded, and converge strongly to $R^{\left(\infty\right)}$, their
square roots converge strongly to $R^{\left(\infty\right)1/2}$. Hence
\[
C_{w}(R^{\left(k_{j}-1\right)})=R^{\left(k_{j}-1\right)1/2}P_{w}R^{\left(k_{j}-1\right)1/2}\xrightarrow{s}R^{\left(\infty\right)1/2}P_{w}R^{\left(\infty\right)1/2}=C_{w}(R^{\left(\infty\right)}).
\]
Therefore $C_{w}(R^{\left(\infty\right)})=0$ for every $w\in\Lambda$.
Using the hypothesis on $\Lambda$ with $S=R^{\left(\infty\right)}$,
we obtain $R^{\left(\infty\right)}=\sum_{w\in\Lambda}C_{w}(R^{\left(\infty\right)})=0$. 
\end{proof}

\begin{cor}
\label{cor:2-5} Let $\left(w_{k}\right)_{k\ge1}$ be a sequence of
nodes in $\mathcal{T}$, and let $\left(R^{\left(k\right)}\right)_{k\ge0}$
be the associated remainder sequence. Suppose there exists $n\ge0$
such that every node $w\in\mathcal{T}$ with $\left|w\right|=n$ occurs
infinitely often in $\left(w_{k}\right)_{k\ge1}$. Then $R^{\left(\infty\right)}=0$. 
\end{cor}

\begin{proof}
Apply \prettyref{thm:2-4} with $\Lambda=\left\{ w\in\mathcal{T}:\left|w\right|=n\right\} $,
using the fixed depth decomposition 
\[
S=\sum_{\left|w\right|=n}C_{w}\left(S\right),\quad S\in B\left(H\right)_{+}.
\]
\end{proof}

\section{Content measures}\label{sec:3}

We write $\partial\mathcal{T}$ for the set of infinite paths in $\mathcal{T}$,
and for each finite word $w\in\mathcal{T}$ we denote by $\left[w\right]\subset\partial\mathcal{T}$
the corresponding cylinder set 
\[
\left[w\right]=\left\{ \omega\in\partial\mathcal{T}:\omega\text{ begins with }w\right\} .
\]

Recall that for $0<q<\infty$, the Schatten class $\mathcal{S}_{q}$
consists of compact operators $A$ whose singular values $\{s_{j}\left(A\right)\}$
satisfy $\sum_{j}s_{j}\left(A\right)^{q}<\infty$; when $A\ge0$ this
is equivalent to $\mathrm{tr}\left(A^{q}\right)<\infty$, and we write
$\left\Vert A\right\Vert ^{q}_{\mathcal{S}_{q}}:=\mathrm{tr}\left(A^{q}\right)$.
Thus $\mathcal{S}_{1}$ is the trace class, and $\mathcal{S}_{2}$
the space of Hilbert-Schmidt (HS) operators. 
\begin{thm}
\label{thm:3-1} Let $R\in\mathcal{S}_{1}\cap B\left(H\right)_{+}$
be a positive trace class operator, and let $\left\{ C_{w}\left(R\right)\right\} _{w\in\mathcal{T}}$
be its packet content operators.
\begin{enumerate}
\item There exists a unique finite Borel measure $\mu_{R}$ on $\partial\mathcal{T}$
such that 
\[
\mu_{R}\left(\left[w\right]\right)=\mathrm{tr}\left(C_{w}\left(R\right)\right)
\]
for every $w\in\mathcal{T}$. In particular $\mu_{R}\left(\partial\mathcal{T}\right)=\mathrm{tr}\left(R\right)$.
\item For each $x\in H$ there exists a function $e_{x}\in L^{1}\left(\partial\mathcal{T},\mu_{R}\right)$
with $e_{x}\ge0$ $\mu_{R}$-almost everywhere such that 
\[
\left\langle x,C_{w}\left(R\right)x\right\rangle =\int_{\left[w\right]}e_{x}\left(\omega\right)d\mu_{R}\left(\omega\right)
\]
for every node $w\in\mathcal{T}$. Also, 
\[
\left\langle x,Rx\right\rangle =\int_{\partial\mathcal{T}}e_{x}\left(\omega\right)d\mu_{R}\left(\omega\right).
\]
\item Moreover the assignment $x\mapsto e_{x}$ is quadratic in the following
sense: for all $x,y\in H$ one has 
\[
e_{x+y}+e_{x-y}=2\left(e_{x}+e_{y}\right)\qquad\text{in }L^{1}\left(\partial\mathcal{T},\mu_{R}\right).
\]
\end{enumerate}
\end{thm}

\begin{proof}
We first construct the boundary measure $\mu_{R}$. For each node
$w\in\mathcal{T}$ define 
\[
\mu^{0}_{R}\left(\left[w\right]\right)=\mathrm{tr}\left(C_{w}\left(R\right)\right).
\]
At the root we have $C_{\emptyset}\left(R\right)=R$ and hence $\mu^{0}_{R}\left(\partial\mathcal{T}\right)=\mu^{0}_{R}\left(\left[\emptyset\right]\right)=\mathrm{tr}\left(R\right)<\infty$.
The splitting identity $C_{w}\left(R\right)=\sum_{v\in\mathrm{ch}(w)}C_{v}\left(R\right)$
and linearity of the trace give 
\[
\mu^{0}_{R}\left(\left[w\right]\right)=\mathrm{tr}\left(C_{w}\left(R\right)\right)=\sum_{v\in\mathrm{ch}(w)}\mathrm{tr}\left(C_{v}\left(R\right)\right)=\sum_{v\in\mathrm{ch}(w)}\mu^{0}_{R}\left(\left[v\right]\right)
\]
for every $w$. Thus the cylinder weights $\mu^{0}_{R}\left(\left[w\right]\right)$
are consistent under refinement along children.

Let $\mathcal{C}$ be the collection of all cylinder sets $\left[w\right]$,
and let $\Sigma$ be the $\sigma$-algebra generated by $\mathcal{C}$.
Any set $A$ which is a finite union of pairwise disjoint cylinders
has a well-defined weight 
\[
\mu^{0}_{R}\left(A\right)=\sum_{\left[w\right]\subset A}\mu^{0}_{R}\left(\left[w\right]\right),
\]
and by consistency this does not depend on the particular choice of
level at which $A$ is written. In this way $\mu^{0}_{R}$ becomes
a finitely additive nonnegative set function on the algebra generated
by $\mathcal{C}$, with total mass $\mu^{0}_{R}(\partial\mathcal{T})=\mathrm{tr}(R)$.

By the standard Kolmogorov/Carathéodory extension there is a unique
finite Borel measure $\mu_{R}$ on $\left(\partial\mathcal{T},\Sigma\right)$
which extends $\mu^{0}_{R}$. This gives the first part of the statement.

We next fix $x\in H$ and construct the corresponding energy measure
and density. For each node $w\in\mathcal{T}$ put 
\[
\nu^{0}_{x}\left(\left[w\right]\right)=\left\langle x,C_{w}\left(R\right)x\right\rangle .
\]
Positivity of $C_{w}\left(R\right)$ implies $\nu^{0}_{x}\left(\left[w\right]\right)\ge0$
for every $w$, and at the root 
\[
\nu^{0}_{x}\left(\left[\emptyset\right]\right)=\left\langle x,C_{\emptyset}\left(R\right)x\right\rangle =\left\langle x,Rx\right\rangle <\infty.
\]
Using again $C_{w}(R)=\sum_{v\in\mathrm{ch}(w)}C_{v}(R)$, we obtain
\[
\nu^{0}_{x}\left(\left[w\right]\right)=\left\langle x,C_{w}(R)x\right\rangle =\sum_{v\in\mathrm{ch}(w)}\left\langle x,C_{v}(R)x\right\rangle =\sum_{v\in\mathrm{ch}(w)}\nu^{0}_{x}(\left[v\right])
\]
for every $w\in\mathcal{T}$. Thus the cylinder weights $\nu^{0}_{x}\left(\left[w\right]\right)$
are also consistent. Extending by finite additivity as before gives
a finitely additive set function $\nu^{0}_{x}$ on the same algebra
as $\mu^{0}_{R}$, and this extends uniquely to a finite Borel measure
$\nu_{x}$ on $\partial\mathcal{T}$ with total mass 
\[
\nu_{x}\left(\partial\mathcal{T}\right)=\nu^{0}_{x}\left(\left[\emptyset\right]\right)=\left\langle x,Rx\right\rangle .
\]

We claim that $\nu_{x}$ is absolutely continuous with respect to
$\mu_{R}$. It is enough to check this on cylinders. Suppose $\mu^{0}_{R}\left(\left[w\right]\right)=0$
for some node $w$. Then $\mathrm{tr}\left(C_{w}\left(R\right)\right)=0$.
Since $C_{w}\left(R\right)\in\mathcal{S}_{1}$ and $C_{w}\left(R\right)\ge0$,
the only way for the trace to vanish is that $C_{w}\left(R\right)\equiv0$.
Consequently 
\[
\nu^{0}_{x}\left(\left[w\right]\right)=\left\langle x,C_{w}\left(R\right)x\right\rangle =0
\]
for every $x$. By finite additivity the same implication holds for
any finite union of cylinders, and by passage to the limit it holds
for all sets in the $\sigma$-algebra generated by cylinders. In other
words $\nu^{0}_{x}\ll\mu^{0}_{R}$ on the cylinder algebra, and therefore
$\nu_{x}\ll\mu_{R}$ on $\left(\partial\mathcal{T},\Sigma\right)$.

Then there exists a unique Radon-Nikodym derivative $e_{x}\in L^{1}(\partial\mathcal{T},\mu_{R})$,
with $e_{x}\ge0$ $\mu_{R}$-a.e., such that 
\[
\nu_{x}\left(A\right)=\int_{A}e_{x}\left(\omega\right)d\mu_{R}\left(\omega\right)
\]
for every Borel set $A\subset\partial\mathcal{T}$. In particular,
for each cylinder $\left[w\right]$ we have 
\[
\left\langle x,C_{w}\left(R\right)x\right\rangle =\nu_{x}\left(\left[w\right]\right)=\int_{\left[w\right]}e_{x}\left(\omega\right)d\mu_{R}\left(\omega\right),
\]
and taking $A=\partial\mathcal{T}$ gives 
\[
\left\langle x,Rx\right\rangle =\nu_{x}\left(\partial\mathcal{T}\right)=\int_{\partial\mathcal{T}}e_{x}\left(\omega\right)d\mu_{R}\left(\omega\right).
\]

It remains to verify the quadratic property. Let $x,y\in H$. For
each node $w$, linearity and selfadjointness of $C_{w}(R)$ give
the familiar parallelogram identity 
\begin{multline*}
\left\langle x+y,C_{w}\left(R\right)\left(x+y\right)\right\rangle +\left\langle x-y,C_{w}\left(R\right)\left(x-y\right)\right\rangle \\
=2\left(\left\langle x,C_{w}\left(R\right)x\right\rangle +\left\langle y,C_{w}\left(R\right)y\right\rangle \right).
\end{multline*}
In terms of the cylinder weights of the measures $\nu^{0}_{z}$ this
says 
\[
\nu^{0}_{x+y}\left(\left[w\right]\right)+\nu^{0}_{x-y}\left(\left[w\right]\right)=2\left(\nu^{0}_{x}\left(\left[w\right]\right)+\nu^{0}_{y}\left(\left[w\right]\right)\right)
\]
for every $w\in\mathcal{T}$. By finite additivity this identity holds
for all sets in the cylinder algebra, and hence for all Borel sets
in $\partial\mathcal{T}$ after passing to the measure extensions.
Thus 
\[
\nu_{x+y}\left(A\right)+\nu_{x-y}\left(A\right)=2\left(\nu_{x}\left(A\right)+\nu_{y}\left(A\right)\right)
\]
for every Borel $A\subset\partial\mathcal{T}$.

Now write everything in terms of the Radon-Nikodym densities: 
\[
\int_{A}\left(e_{x+y}+e_{x-y}-2e_{x}-2e_{y}\right)d\mu_{R}=0
\]
for every Borel set $A$. Since $\mu_{R}$ is finite and $A$ is arbitrary,
this implies 
\[
e_{x+y}+e_{x-y}=2\left(e_{x}+e_{y}\right)
\]
in $L^{1}\left(\partial\mathcal{T},\mu_{R}\right)$, as claimed.
\end{proof}

\begin{rem*}
Formally, the family $\{C_{w}(R)\}$ may be viewed as the values of
a positive operator-valued measure on the packet boundary, obtained
by compressing a projection-valued measure $E$ by $R^{1/2}$. We
will not develop this here.
\end{rem*}
The densities $e_{x}$ from \prettyref{thm:3-1} may be recovered
from the finite depth packet data. For $\left|w\right|=n$, normalize
$\left\langle x,C_{w}\left(R\right)x\right\rangle $ by the cylinder
mass $\mu_{R}\left(\left[w\right]\right)$. This gives a piecewise
constant function $E^{x}_{n}$ on the depth $n$ cylinder partition
of $\partial\mathcal{T}$. As the partition is refined, the functions
$E^{x}_{n}$ form a martingale and converge to $e_{x}$.
\begin{thm}
\label{thm:3-2} Let $R\in\mathcal{S}_{1}\cap B\left(H\right)_{+}$,
let $\mu_{R}$ be the boundary measure from \prettyref{thm:3-1},
and for each $n\ge0$ let $\mathcal{F}_{n}$ be the $\sigma$-algebra
generated by the cylinders $\left[w\right]$ with $\left|w\right|=n$.
For each $x\in H$, define the $\mathcal{F}_{n}$-measurable function
$E^{x}_{n}$ on $\partial\mathcal{T}$ by 
\[
E^{x}_{n}\left(\omega\right)=\frac{\left\langle x,C_{w}\left(R\right)x\right\rangle }{\mu_{R}\left(\left[w\right]\right)}\qquad\text{for }\omega\in\left[w\right],\ \left|w\right|=n,
\]
with the convention that $E^{x}_{n}\left(\omega\right)=0$ on $\left[w\right]$
whenever $\mu_{R}\left(\left[w\right]\right)=0$. Then 
\begin{enumerate}
\item For every node $w$ with $\left|w\right|=n$, 
\[
\int_{\left[w\right]}E^{x}_{n}\left(\omega\right)d\mu_{R}\left(\omega\right)=\left\langle x,C_{w}\left(R\right)x\right\rangle .
\]
In particular, $E^{x}_{n}\in L^{1}\left(\partial\mathcal{T},\mu_{R}\right)$
and 
\[
\int_{\partial\mathcal{T}}E^{x}_{n}\left(\omega\right)d\mu_{R}\left(\omega\right)=\left\langle x,Rx\right\rangle .
\]
\item $\left(E^{x}_{n}\right)_{n\ge0}$ is a nonnegative martingale with
respect to the filtration $\left(\mathcal{F}_{n}\right)_{n\ge0}$. 
\item If $e_{x}$ is the density from \prettyref{thm:3-1}, then 
\[
E^{x}_{n}\to e_{x}
\]
almost everywhere and in $L^{1}\left(\partial\mathcal{T},\mu_{R}\right)$
as $n\to\infty$. 
\end{enumerate}
\end{thm}

\begin{proof}
For each $n$, the function $E^{x}_{n}$ is constant on every cylinder
$\left[w\right]$ with $\left|w\right|=n$, hence is $\mathcal{F}_{n}$-measurable.
If $\mu_{R}\left(\left[w\right]\right)=0$, then by \prettyref{thm:3-1}
one also has 
\[
\left\langle x,C_{w}\left(R\right)x\right\rangle =\int_{\left[w\right]}e_{x}\left(\omega\right)d\mu_{R}\left(\omega\right)=0,
\]
so the stated identity on $\left[w\right]$ is still valid. If $\mu_{R}\left(\left[w\right]\right)>0$,
then 
\[
\int_{\left[w\right]}E^{x}_{n}\left(\omega\right)d\mu_{R}\left(\omega\right)=\frac{\left\langle x,C_{w}\left(R\right)x\right\rangle }{\mu_{R}\left(\left[w\right]\right)}\mu_{R}\left(\left[w\right]\right)=\left\langle x,C_{w}\left(R\right)x\right\rangle .
\]
This proves (1), and summing over all cylinders at depth $n$ gives
\[
\int_{\partial\mathcal{T}}E^{x}_{n}\left(\omega\right)d\mu_{R}\left(\omega\right)=\sum_{\left|w\right|=n}\left\langle x,C_{w}\left(R\right)x\right\rangle =\left\langle x,Rx\right\rangle .
\]

We next prove the martingale property. Fix $n\ge0$ and let $w$ be
a node with $\left|w\right|=n$. Since $E^{x}_{n+1}$ is constant
on each child cylinder $\left[v\right]$, one has 
\[
\int_{\left[w\right]}E^{x}_{n+1}\left(\omega\right)d\mu_{R}\left(\omega\right)=\sum_{v\in\mathrm{ch}\left(w\right)}\int_{\left[v\right]}E^{x}_{n+1}\left(\omega\right)d\mu_{R}\left(\omega\right).
\]
By part (1), this equals 
\[
\sum_{v\in\mathrm{ch}\left(w\right)}\left\langle x,C_{v}\left(R\right)x\right\rangle =\left\langle x,C_{w}\left(R\right)x\right\rangle =\int_{\left[w\right]}E^{x}_{n}\left(\omega\right)d\mu_{R}\left(\omega\right),
\]
where we used the refinement identity \prettyref{eq:2-3}. Since both
$E^{x}_{n}$ and $\mathbb{E}\left(E^{x}_{n+1}\mid\mathcal{F}_{n}\right)$
are $\mathcal{F}_{n}$-measurable and have the same integral over
each atom $\left[w\right]$ of $\mathcal{F}_{n}$, they agree $\mu_{R}$-almost
everywhere. Thus 
\[
\mathbb{E}\left(E^{x}_{n+1}\mid\mathcal{F}_{n}\right)=E^{x}_{n},
\]
so $\left(E^{x}_{n}\right)_{n\ge0}$ is a nonnegative martingale.

We now identify $E^{x}_{n}$ as the conditional expectation of $e_{x}$.
For every node $w$ with $\left|w\right|=n$, part (1) and \prettyref{thm:3-1}
give 
\[
\int_{\left[w\right]}E^{x}_{n}\left(\omega\right)d\mu_{R}\left(\omega\right)=\left\langle x,C_{w}\left(R\right)x\right\rangle =\int_{\left[w\right]}e_{x}\left(\omega\right)d\mu_{R}\left(\omega\right).
\]
Since the cylinders $\left[w\right]$ with $\left|w\right|=n$ are
the atoms of $\mathcal{F}_{n}$, it follows that 
\[
\mathbb{E}\left(e_{x}\mid\mathcal{F}_{n}\right)=E^{x}_{n}.
\]
Therefore, by the martingale convergence theorem for conditional expectations,
\[
E^{x}_{n}\to e_{x}
\]
almost everywhere and in $L^{1}\left(\partial\mathcal{T},\mu_{R}\right)$
as $n\to\infty$. This proves (3). 
\end{proof}

\section{Trace class greedy extraction}\label{sec:4}

In this section we fix a depth $n$ and study the trace weights $\mathrm{tr}\left(C_{w}\left(R\right)\right)$
under repeated removal of the largest depth $n$ block. The argument
uses only the fixed depth decomposition from \prettyref{eq:4-1}.
The boundary measure from \prettyref{sec:3} gives an equivalent interpretation
of the rule as selecting the heaviest cylinder at that level, as noted
in \prettyref{rem:4-2}. We show that this fixed depth rule yields
a geometric decay estimate for the trace of the remainder, with a
rate depending only on the number $N_{n}$ of packets at depth $n$.

We assume throughout this section that $R\in\mathcal{S}_{1}\cap B\left(H\right)_{+}$,
so that all traces are finite.

For each depth $n\ge1$ we write 
\[
\mathcal{T}_{n}=\left\{ w\in\mathcal{T}:\left|w\right|=n\right\} ,\qquad N_{n}=\#\mathcal{T}_{n}.
\]
At depth $n$ we have the orthogonal decomposition 
\[
I=\sum_{w\in\mathcal{T}_{n}}P_{w},
\]
and therefore for every $A\in B\left(H\right)_{+}$, 
\begin{equation}
A=\sum_{w\in\mathcal{T}_{n}}C_{w}\left(A\right).\label{eq:4-1}
\end{equation}

Fix a depth $n\ge1$. Starting from $R$, at each step we remove a
depth $n$ content block of maximal trace.

We define the depth $n$ greedy extraction sequence starting from
$R$ by 
\[
R^{(0)}=R.
\]
Given $R^{(k-1)}$ at step $k\ge1$, we choose a node $w_{k}\in\mathcal{T}_{n}$
such that 
\[
\mathrm{tr}(C_{w_{k}}(R^{(k-1)}))=\max_{w\in\mathcal{T}_{n}}\mathrm{tr}(C_{w}(R^{(k-1)})),
\]
define the extracted piece 
\[
D_{k}=C_{w_{k}}(R^{(k-1)}),
\]
and update 
\[
R^{(k)}=R^{(k-1)}-D_{k}.
\]
The maximum exists at each step because $\mathcal{T}_{n}$ is finite
and each $\mathrm{tr}\left(C_{w}\left(R^{(k-1)}\right)\right)$ is
nonnegative. By \prettyref{lem:2-2}, the sequence $\left(R^{(k)}\right)_{k\ge0}$
decreases in the Loewner order, converges strongly to a limit $R^{(\infty)}\in B\left(H\right)_{+}$,
and for every $N\ge1$, \eqref{eq:2-4-1} holds.

The construction is now the same as the sequential extraction process
from \prettyref{sec:2}, except that the node choices are restricted
to a fixed depth. The next theorem shows that this restriction leads
to a geometric decay estimate in trace norm, with rate depending only
on $N_{n}$.
\begin{thm}
\label{thm:4-1} Let $R\in\mathcal{S}_{1}\cap B\left(H\right)_{+}$,
fix a depth $n\ge1$, and let $\left(R^{(k)},D_{k}\right)_{k\ge1}$
be the depth $n$ greedy extraction sequence defined above. Then for
every $k\ge0$ we have 
\[
\mathrm{tr}(R^{(k)})\le\left(1-\frac{1}{N_{n}}\right)^{k}\mathrm{tr}\left(R\right).
\]
In particular $R^{(k)}\to0$ in trace norm as $k\to\infty$, and 
\[
R=\sum^{\infty}_{k=1}D_{k}
\]
with convergence in $\mathcal{S}_{1}$. More precisely, for each $m\ge1$,
\[
\left\Vert R-\sum\nolimits^{m}_{k=1}D_{k}\right\Vert _{1}=\Vert R^{(m)}\Vert_{1}\le\left(1-\frac{1}{N_{n}}\right)^{m}\mathrm{tr}\left(R\right).
\]
\end{thm}

\begin{proof}
Fix $k\ge1$. Apply \eqref{eq:4-1} to $R^{(k-1)}$ to get 
\[
R^{(k-1)}=\sum_{w\in\mathcal{T}_{n}}C_{w}(R^{(k-1)}).
\]
Taking traces and using linearity of $\mathrm{tr}$ on $\mathcal{S}_{1}$,
we obtain 
\[
\mathrm{tr}(R^{(k-1)})=\sum_{w\in\mathcal{T}_{n}}\mathrm{tr}(C_{w}(R^{(k-1)})).
\]
The right-hand side is a sum of $N_{n}$ nonnegative numbers. Therefore
the maximal term is at least the average, so 
\[
\mathrm{tr}\left(D_{k}\right)=\max_{w\in\mathcal{T}_{n}}\mathrm{tr}(C_{w}(R^{(k-1)}))\ge\frac{1}{N_{n}}\mathrm{tr}(R^{(k-1)}).
\]
On the other hand, by definition of $R^{(k)}$ and positivity, 
\begin{align*}
\mathrm{tr}(R^{(k)}) & =\mathrm{tr}(R^{(k-1)})-\mathrm{tr}\left(D_{k}\right)\le\left(1-\frac{1}{N_{n}}\right)\mathrm{tr}(R^{(k-1)}).
\end{align*}
This proves the one-step inequality 
\[
\mathrm{tr}(R^{(k)})\le\left(1-\frac{1}{N_{n}}\right)\mathrm{tr}(R^{(k-1)})
\]
for every $k\ge1$. Iterating in $k$ yields 
\[
\mathrm{tr}(R^{(k)})\le\left(1-\frac{1}{N_{n}}\right)^{k}\mathrm{tr}\left(R^{(0)}\right)=\left(1-\frac{1}{N_{n}}\right)^{k}\mathrm{tr}\left(R\right),
\]
as claimed.

Since each $R^{(k)}$ is positive, its trace norm agrees with its
trace $\Vert R^{(k)}\Vert_{1}=\mathrm{tr}\left((R^{(k)})\right)$.
The estimate above therefore shows that $\Vert R^{(k)}\Vert_{1}\to0$
as $k\to\infty$, i.e., $R^{(k)}\to0$ in $\mathcal{S}_{1}$.

Using \eqref{eq:2-4-1}, we get 
\[
\left\Vert R-\sum\nolimits^{m}_{k=1}D_{k}\right\Vert _{1}=\Vert R^{(m)}\Vert_{1}=\mathrm{tr}(R^{(m)})\le\left(1-\frac{1}{N_{n}}\right)^{m}\mathrm{tr}\left(R\right).
\]
The right-hand side tends to zero as $m\to\infty$, so $\sum_{k\ge1}D_{k}$
converges to $R$ in trace norm. 
\end{proof}

\begin{rem}
\label{rem:4-2} The boundary formulation from \prettyref{thm:3-1}
gives a convenient description of the depth $n$ greedy rule. Since
each $R^{(k)}$ is positive trace class, it has an associated finite
Borel measure $\mu_{R^{(k)}}$ on $\partial\mathcal{T}$ satisfying
\[
\mu_{R^{(k)}}\left(\left[w\right]\right)=\mathrm{tr}(C_{w}(R^{(k)}))\qquad\text{for every }w\in\mathcal{T},
\]
and in particular 
\[
\mu_{R^{(k)}}\left(\partial\mathcal{T}\right)=\mathrm{tr}(R^{(k)}).
\]
Thus, at step $k$, the depth $n$ greedy choice is precisely 
\[
w_{k}\in\arg\max_{\left|w\right|=n}\mu_{R^{(k-1)}}\left(\left[w\right]\right),
\]
and the extracted mass is 
\[
\mathrm{tr}\left(D_{k}\right)=\mu_{R^{(k-1)}}\left(\left[w_{k}\right]\right).
\]
Accordingly, \prettyref{thm:4-1} may be read as a geometric decay
statement for the total boundary mass under the rule that, at each
step, removes the heaviest depth $n$ cylinder. 
\end{rem}

The estimate in \prettyref{thm:4-1} is a worst case bound depending
only on the number $N_{n}$ of depth $n$ packets. It is natural to
ask whether the factor $1-\frac{1}{N_{n}}$ can be improved at this
level of generality. The next proposition shows that it cannot: the
bound in \prettyref{thm:4-1} is sharp among geometric decay estimates
whose rate depends only on $N_{n}$.
\begin{prop}
\label{prop:4-3} The rate in \prettyref{thm:4-1} is optimal in general.
More precisely, for every depth $n\ge1$ there exists $R\in\mathcal{S}_{1}\cap B\left(H\right)_{+}$
such that the associated depth $n$ greedy extraction sequence satisfies
\[
R^{(k)}=\left(1-\frac{1}{N_{n}}\right)^{k}R
\]
for all $k\ge0$. Consequently, 
\[
\mathrm{tr}\left(R^{(k)}\right)=\left(1-\frac{1}{N_{n}}\right)^{k}\mathrm{tr}\left(R\right)
\]
for all $k\ge0$, so no better universal decay factor can replace
$1-\frac{1}{N_{n}}$ in \prettyref{thm:4-1}. 
\end{prop}

\begin{proof}
For each $w\in\mathcal{T}_{n}$ choose a unit vector $e_{w}\in W_{w}$.
Since the subspaces $\left\{ W_{w}:\left|w\right|=n\right\} $ are
orthogonal, the family $\left\{ e_{w}:w\in\mathcal{T}_{n}\right\} $
is orthonormal. 

Set $u=\sum_{w\in\mathcal{T}_{n}}e_{w}$, $R=\left|u\left\rangle \right\langle u\right|$.
Then $\left\Vert u\right\Vert ^{2}=N_{n}$. Let $p=\frac{1}{N_{n}}\left|u\left\rangle \right\langle u\right|$,
so $p$ is the rank one projection onto $\mathbb{C}u$, and $R=N_{n}p$,
$R^{1/2}=\sqrt{N_{n}}p$. 

Fix $w\in\mathcal{T}_{n}$. Since $P_{w}u=e_{w}$ and $\left\langle u,P_{w}u\right\rangle =\left\Vert e_{w}\right\Vert ^{2}=1$,
we get 
\begin{align*}
pP_{w}p & =\frac{1}{N^{2}_{n}}\left|u\left\rangle \right\langle u\right|P_{w}\left|u\left\rangle \right\langle u\right|\\
 & =\frac{1}{N^{2}_{n}}\left\langle u,P_{w}u\right\rangle \left|u\left\rangle \right\langle u\right|=\frac{1}{N^{2}_{n}}\left|u\left\rangle \right\langle u\right|=\frac{1}{N_{n}}p.
\end{align*}
Hence $C_{w}\left(R\right)=R^{1/2}P_{w}R^{1/2}=N_{n}pP_{w}p=p=\frac{1}{N_{n}}R$
for every $w\in\mathcal{T}_{n}$. Thus all depth $n$ blocks have
the same trace, so every depth $n$ node is a greedy choice. Therefore
$D_{1}=\frac{1}{N_{n}}R$, $R^{(1)}=R-D_{1}=\left(1-\frac{1}{N_{n}}\right)R$. 

Now suppose $R^{(k-1)}=\left(1-\frac{1}{N_{n}}\right)^{k-1}R$. Since
$C_{w}\left(\alpha R\right)=\alpha C_{w}\left(R\right)$ for all $\alpha\ge0$,
we obtain 
\[
C_{w}\left(R^{(k-1)}\right)=\left(1-\frac{1}{N_{n}}\right)^{k-1}C_{w}\left(R\right)=\frac{1}{N_{n}}R^{(k-1)}
\]
for every $w\in\mathcal{T}_{n}$. Again every depth $n$ node is a
greedy choice, and 
\[
R^{(k)}=R^{(k-1)}-\frac{1}{N_{n}}R^{(k-1)}=\left(1-\frac{1}{N_{n}}\right)R^{(k-1)}.
\]
By induction, $R^{(k)}=\left(1-\frac{1}{N_{n}}\right)^{k}R$ for all
$k\ge0$. Taking traces gives the stated equality. 
\end{proof}

\section{Hilbert-Schmidt greedy extraction}\label{sec:5}

\prettyref{sec:4} treats the depth $n$ packet blocks only through
their trace weights. In the Hilbert-Schmidt (HS) setting, the relative
arrangement of those blocks also matters. This leads to a different
fixed depth selection rule, based on the quantities $\left\Vert C_{w}\left(A\right)\right\Vert _{2}$,
and to a depth $n$ coherence parameter $\Gamma_{n}\left(A\right)$
measuring how far $A$ is from being block diagonal with respect to
the depth $n$ packet partition. Here we show that the resulting HS
greedy rule again yields geometric decay of the remainder, with an
improved one-step factor when the operator is closer to block diagonal
at depth $n$.

Throughout we assume that $R\in\mathcal{S}_{1}\cap B\left(H\right)_{+}$,
so in particular $R\in\mathcal{S}_{2}$. We define 
\[
E_{n}\left(X\right)=\sum_{w\in\mathcal{T}_{n}}P_{w}XP_{w},
\]
i.e., the conditional expectation onto the block diagonal subalgebra
determined by $\left(P_{w}\right)_{w\in\mathcal{T}_{n}}$. 

For any $A\in\mathcal{S}_{2}\cap B\left(H\right)_{+}$ we set 
\[
\Gamma_{n}\left(A\right)=\frac{\left\Vert A\right\Vert ^{2}_{2}}{\sum_{w\in\mathcal{T}_{n}}\left\Vert C_{w}\left(A\right)\right\Vert ^{2}_{2}},
\]
whenever the denominator is nonzero. This parameter will serve as
a ``depth $n$ coherence'' of $A$ relative to the packet decomposition.
\begin{lem}
\label{lem:5-1}Let $A\in\mathcal{S}_{2}\cap B\left(H\right)_{+}$.
Then
\begin{equation}
\sum_{w\in\mathcal{T}_{n}}\left\Vert C_{w}\left(A\right)\right\Vert ^{2}_{2}=\mathrm{tr}\left(E_{n}\left(A\right)A\right).\label{eq:5-1}
\end{equation}
Moreover, 
\begin{equation}
\frac{1}{N_{n}}\left\Vert A\right\Vert ^{2}_{2}\le\sum_{w\in\mathcal{T}_{n}}\left\Vert C_{w}\left(A\right)\right\Vert ^{2}_{2}\le\left\Vert A\right\Vert ^{2}_{2},\label{eq:5-2}
\end{equation}
and hence $1\le\Gamma_{n}\left(A\right)\le N_{n}$.

In particular, $\Gamma_{n}\left(A\right)=1$ if and only if $E_{n}\left(A\right)=A$,
equivalently if and only if $P_{w}A=AP_{w}$ for all $w\in\mathcal{T}_{n}$
(i.e., $A$ is block diagonal at depth $n$).
\end{lem}

\begin{proof}
Note that $C_{w}\left(A\right)=A^{1/2}P_{w}A^{1/2}$, so that 
\[
C_{w}\left(A\right)^{2}=A^{1/2}P_{w}AP_{w}A^{1/2}\in\mathcal{S}_{1}\cap B\left(H\right)_{+}.
\]
Thus, 
\[
\left\Vert C_{w}\left(A\right)\right\Vert ^{2}_{2}=\mathrm{tr}\left(C_{w}\left(A\right)^{2}\right)=\mathrm{tr}\left(P_{w}AP_{w}A\right).
\]
Summing over $w\in\mathcal{T}_{n}$ gives \eqref{eq:5-1}.

View $\mathcal{S}_{2}$ as a Hilbert space with inner product $\left\langle X,Y\right\rangle _{2}=\mathrm{tr}\left(Y^{*}X\right)$.
Then $E_{n}$ is selfadjoint and idempotent on $\mathcal{S}_{2}$,
because for $X,Y\in\mathcal{S}_{2}$, 
\begin{align*}
\left\langle E_{n}\left(X\right),Y\right\rangle _{2} & =\mathrm{tr}\left(E_{n}\left(X\right)^{*}Y\right)\\
 & =\sum_{w}\mathrm{tr}\left(P_{w}X^{*}P_{w}Y\right)=\mathrm{tr}\left(X^{*}E_{n}\left(Y\right)\right)=\left\langle X,E_{n}\left(Y\right)\right\rangle _{2}.
\end{align*}
Thus $E_{n}$ is the orthogonal projection onto the subspace of block
diagonal operators, in particular 
\[
\left\Vert E_{n}\left(X\right)\right\Vert _{2}\le\left\Vert X\right\Vert _{2}
\]
for all $X\in\mathcal{S}_{2}$. Using Cauchy-Schwarz and \eqref{eq:5-1},
we get 
\begin{equation}
\sum_{w\in\mathcal{T}_{n}}\left\Vert C_{w}\left(A\right)\right\Vert ^{2}_{2}=\mathrm{tr}\left(E_{n}\left(A\right)A\right)\le\left\Vert E_{n}\left(A\right)\right\Vert _{2}\left\Vert A\right\Vert _{2}\le\left\Vert A\right\Vert ^{2}_{2}.\label{eq:5-3}
\end{equation}

For the lower bound, use \eqref{eq:4-1} in the Hilbert space $\mathcal{S}_{2}$.
Then 
\[
\left\Vert A\right\Vert ^{2}_{2}=\left\Vert \sum\nolimits_{w\in\mathcal{T}_{n}}C_{w}\left(A\right)\right\Vert ^{2}_{2}\le N_{n}\sum\nolimits_{w\in\mathcal{T}_{n}}\left\Vert C_{w}\left(A\right)\right\Vert ^{2}_{2}
\]
by the elementary inequality $\left\Vert \sum^{N}_{i=1}x_{i}\right\Vert ^{2}\le N\sum_{i}\left\Vert x_{i}\right\Vert ^{2}$.
This gives 
\begin{equation}
\sum_{w\in\mathcal{T}_{n}}\left\Vert C_{w}\left(A\right)\right\Vert ^{2}_{2}\ge\frac{1}{N_{n}}\left\Vert A\right\Vert ^{2}_{2}.\label{eq:5-4}
\end{equation}
Combining \eqref{eq:5-3}-\eqref{eq:5-4} gives \eqref{eq:5-2}, and
so $1\le\Gamma_{n}\left(A\right)\le N_{n}$.

Finally, if $E_{n}\left(A\right)=A$, then 
\[
\sum_{w\in\mathcal{T}_{n}}\left\Vert C_{w}\left(A\right)\right\Vert ^{2}_{2}=\mathrm{tr}\left(E_{n}\left(A\right)A\right)=\mathrm{tr}\left(A^{2}\right)=\left\Vert A\right\Vert ^{2}_{2},
\]
so $\Gamma_{n}\left(A\right)=1$. Conversely, if $\Gamma_{n}\left(A\right)=1$
then equality holds in the chain below:
\[
\mathrm{tr}\left(E_{n}\left(A\right)A\right)\le\left\Vert E_{n}\left(A\right)\right\Vert _{2}\left\Vert A\right\Vert _{2}\le\left\Vert A\right\Vert ^{2}_{2}.
\]
Thus $\left\Vert E_{n}\left(A\right)\right\Vert _{2}=\left\Vert A\right\Vert _{2}$,
and since $E_{n}$ is the orthogonal projection on $\mathcal{S}_{2}$,
this forces $E_{n}\left(A\right)=A$. The equivalence with commuting
with every $P_{w}$ at depth $n$ is standard. 
\end{proof}

\begin{lem}
\label{lem:5-2}Let $A,D\in\mathcal{S}_{2}\cap B\left(H\right)_{+}$
with $0\le D\le A$. Then 
\[
\left\Vert A-D\right\Vert ^{2}_{2}\le\left\Vert A\right\Vert ^{2}_{2}-\left\Vert D\right\Vert ^{2}_{2}.
\]
\end{lem}

\begin{proof}
Write $B=A-D\ge0$. Then, 
\[
\left\Vert A-D\right\Vert ^{2}_{2}=\left\Vert A\right\Vert ^{2}_{2}+\left\Vert D\right\Vert ^{2}_{2}-2\mathrm{tr}\left(AD\right).
\]
It suffices to show $\mathrm{tr}\left(AD\right)\ge\mathrm{tr}\left(D^{2}\right)$.
We can write 
\[
\mathrm{tr}\left(AD\right)-\mathrm{tr}\left(D^{2}\right)=\mathrm{tr}\left(\left(A-D\right)D\right)=\mathrm{tr}\left(BD\right).
\]
Since $B,D\ge0$, we have $D^{1/2}BD^{1/2}\ge0$. Moreover, $D^{1/2}\in\mathcal{S}_{4}$
and $B\in\mathcal{S}_{2}$, so $D^{1/2}BD^{1/2}\in\mathcal{S}_{1}$
and its trace is nonnegative: 
\[
\mathrm{tr}\left(BD\right)=\mathrm{tr}\left(D^{1/2}BD^{1/2}\right)\ge0.
\]
This gives $\mathrm{tr}\left(AD\right)\ge\mathrm{tr}\left(D^{2}\right)$. 
\end{proof}

At depth $n$, starting from $R^{\left(0\right)}=R$, we define inductively
\[
w_{k}\in\mathcal{T}_{n}\quad\text{such that}\quad\Vert C_{w_{k}}(R^{\left(k-1\right)})\Vert_{2}=\max_{w\in\mathcal{T}_{n}}\left\Vert C_{w}(R^{\left(k-1\right)})\right\Vert _{2},
\]
then set 
\[
D_{k}=C_{w_{k}}(R^{\left(k-1\right)}),\qquad R^{\left(k\right)}=R^{\left(k-1\right)}-D_{k}.
\]
As in the trace greedy case, each $D_{k}$ is positive, satisfies
$0\le D_{k}\le R^{\left(k-1\right)}$, and is trace class; hence all
$R^{\left(k\right)}$ remain in $\mathcal{S}_{1}\cap B\left(H\right)_{+}$,
and in particular in $\mathcal{S}_{2}$.
\begin{thm}
\label{thm:5-3}Let $R\in\mathcal{S}_{1}\cap B\left(H\right)_{+}$,
fix $n\ge1$, and let $\left(R^{\left(k\right)}\right)$ be the depth
$n$ HS greedy sequence from above. Then for every $k\ge1$, 
\begin{equation}
\Vert R^{\left(k\right)}\Vert^{2}_{2}\le\left(1-\frac{1}{\Gamma_{n}\left(R^{\left(k-1\right)}\right)N_{n}}\right)\Vert R^{\left(k-1\right)}\Vert^{2}_{2}.\label{eq:5-5}
\end{equation}
In particular, since $\Gamma_{n}\left(R^{\left(k-1\right)}\right)\le N_{n}$,
one has the uniform estimate 
\begin{equation}
\Vert R^{\left(k\right)}\Vert^{2}_{2}\le\left(1-\frac{1}{N^{2}_{n}}\right)^{k}\left\Vert R\right\Vert ^{2}_{2},\label{eq:5-6}
\end{equation}
and therefore $R^{\left(k\right)}\to0$ in $\mathcal{S}_{2}$. 

Moreover, $R=\sum^{\infty}_{k=1}D_{k}$ with convergence in $\mathcal{S}_{2}$,
and for each $m\ge1$ 
\begin{equation}
\left\Vert R-\sum\nolimits^{m}_{k=1}D_{k}\right\Vert _{2}=\Vert R^{\left(m\right)}\Vert_{2}\le\left(1-\frac{1}{N^{2}_{n}}\right)^{m/2}\left\Vert R\right\Vert _{2}.\label{eq:5-7}
\end{equation}

Finally, if $R^{\left(k-1\right)}$ is block diagonal at depth $n$,
i.e. $\Gamma_{n}\left(R^{\left(k-1\right)}\right)=1$, then the one-step
factor improves to 
\begin{equation}
\Vert R^{\left(k\right)}\Vert^{2}_{2}\le\left(1-\frac{1}{N_{n}}\right)\Vert R^{\left(k-1\right)}\Vert^{2}_{2}.\label{eq:5-8}
\end{equation}
\end{thm}

\begin{proof}
Fix $k\ge1$ and set $A=R^{\left(k-1\right)}$. By construction, $D_{k}=C_{w_{k}}\left(A\right)$
for some $w_{k}\in\mathcal{T}_{n}$, and we have $0\le D_{k}\le A$.
\prettyref{lem:5-2} then gives 
\begin{equation}
\Vert R^{\left(k\right)}\Vert^{2}_{2}=\left\Vert A-D_{k}\right\Vert ^{2}_{2}\le\left\Vert A\right\Vert ^{2}_{2}-\left\Vert D_{k}\right\Vert ^{2}_{2}.\label{eq:5-9}
\end{equation}
By the choice of $w_{k}$, the block $D_{k}$ has maximal HS norm
among $\left\{ C_{w}\left(A\right):w\in\mathcal{T}_{n}\right\} $.
Hence 
\[
\left\Vert D_{k}\right\Vert ^{2}_{2}=\max_{w\in\mathcal{T}_{n}}\left\Vert C_{w}\left(A\right)\right\Vert ^{2}_{2}\ge\frac{1}{N_{n}}\sum_{w\in\mathcal{T}_{n}}\left\Vert C_{w}\left(A\right)\right\Vert ^{2}_{2}.
\]
By the definition of $\Gamma_{n}\left(A\right)$, 
\[
\sum_{w\in\mathcal{T}_{n}}\left\Vert C_{w}\left(A\right)\right\Vert ^{2}_{2}=\frac{\left\Vert A\right\Vert ^{2}_{2}}{\Gamma_{n}\left(A\right)}.
\]
Putting these together, 
\[
\left\Vert D_{k}\right\Vert ^{2}_{2}\ge\frac{1}{N_{n}}\cdot\frac{\left\Vert A\right\Vert ^{2}_{2}}{\Gamma_{n}\left(A\right)}=\frac{1}{\Gamma_{n}\left(A\right)N_{n}}\left\Vert A\right\Vert ^{2}_{2}.
\]
Substituting into \eqref{eq:5-9} yields 
\[
\Vert R^{\left(k\right)}\Vert^{2}_{2}\le\left(1-\frac{1}{\Gamma_{n}\left(A\right)N_{n}}\right)\left\Vert A\right\Vert ^{2}_{2}=\left(1-\frac{1}{\Gamma_{n}\left(R^{\left(k-1\right)}\right)N_{n}}\right)\Vert R^{\left(k-1\right)}\Vert^{2}_{2},
\]
which is the first assertion \eqref{eq:5-5}.

The uniform bound follows from $\Gamma_{n}\left(R^{\left(k-1\right)}\right)\le N_{n}$,
see \prettyref{lem:5-1}. This gives 
\[
\Vert R^{\left(k\right)}\Vert^{2}_{2}\le\left(1-\frac{1}{N^{2}_{n}}\right)\Vert R^{\left(k-1\right)}\Vert^{2}_{2}
\]
for all $k$, and iterating yields \eqref{eq:5-6}.

In particular, $R^{\left(k\right)}\to0$ in $\mathcal{S}_{2}$. Since
$R^{\left(m\right)}=R-\sum^{m}_{k=1}D_{k}$ by construction, this
is the same as saying that $\sum^{m}_{k=1}D_{k}\to R$ in $\mathcal{S}_{2}$,
with the remainder bound stated in \eqref{eq:5-7}.

The final improvement \eqref{eq:5-8} in the block diagonal case is
immediate from the proof: if $R^{\left(k-1\right)}$ is block diagonal
at depth $n$, then $\Gamma_{n}\left(R^{\left(k-1\right)}\right)=1$
by \prettyref{lem:5-1}, and the factor becomes $1-1/N_{n}$. 
\end{proof}

\begin{rem}
The depth $n$ HS rule should be compared with the trace based rule
from \prettyref{sec:4}. In the trace class setting, each $R^{\left(k\right)}$
determines a boundary measure $\mu_{R^{\left(k\right)}}$ on $\partial\mathcal{T}$,
and the depth $n$ trace greedy rule selects a cylinder of maximal
$\mu_{R^{\left(k-1\right)}}$-mass. The HS rule uses the same packet
decomposition, but now the quantity $\left\Vert C_{w}\left(R^{\left(k-1\right)}\right)\right\Vert _{2}$
depends also on how $R^{\left(k-1\right)}$ sits relative to the block
diagonal decomposition at that depth. The parameter $\Gamma_{n}\left(R^{\left(k-1\right)}\right)$
measures this effect. In particular, $\Gamma_{n}\left(R^{\left(k-1\right)}\right)=1$
exactly when $R^{\left(k-1\right)}$ is block diagonal at depth $n$,
and in that case the one-step factor in \prettyref{thm:5-3} improves
to $1-\frac{1}{N_{n}}$. 
\end{rem}

The parameter $\Gamma_{n}\left(A\right)$ depends on the depth $n$
packet partition, and it is natural to ask how it behaves under refinement.
The next proposition shows that coherence can only increase with depth.
Equivalently, the total squared HS content can only decrease as the
packet partition is refined, and the loss is exactly measured by the
interaction between sibling contents.
\begin{prop}
\label{prop:5-5} Let $A\in\mathcal{S}_{2}\cap B\left(H\right)_{+}$.
For each $n\ge0$ set 
\[
M_{n}\left(A\right):=\sum_{\left|w\right|=n}\left\Vert C_{w}\left(A\right)\right\Vert ^{2}_{2}.
\]
Then $M_{n+1}\left(A\right)\le M_{n}\left(A\right)$ for every $n\ge0$.
Equivalently, whenever $\Gamma_{n}\left(A\right)$ and $\Gamma_{n+1}\left(A\right)$
are defined, one has 
\[
\Gamma_{n}\left(A\right)\le\Gamma_{n+1}\left(A\right).
\]

More precisely, 
\[
M_{n}\left(A\right)-M_{n+1}\left(A\right)=\sum_{\left|w\right|=n}\sum_{\substack{v,v'\in\mathrm{ch}\left(w\right)\\
v\ne v'
}
}\mathrm{tr}\left(C_{v}\left(A\right)C_{v'}\left(A\right)\right).
\]
\end{prop}

\begin{proof}
Fix $w\in\mathcal{T}$ with $\left|w\right|=n$. By \prettyref{eq:2-3},
$C_{w}\left(A\right)=\sum_{v\in\mathrm{ch}\left(w\right)}C_{v}\left(A\right)$.
Therefore 
\[
\left\Vert C_{w}\left(A\right)\right\Vert ^{2}_{2}=\sum_{v\in\mathrm{ch}\left(w\right)}\left\Vert C_{v}\left(A\right)\right\Vert ^{2}_{2}+\sum_{\substack{v,v'\in\mathrm{ch}\left(w\right)\\
v\ne v'
}
}\mathrm{tr}\left(C_{v}\left(A\right)C_{v'}\left(A\right)\right).
\]
Since each $C_{v}\left(A\right)$ is positive, every cross term $\mathrm{tr}\left(C_{v}\left(A\right)C_{v'}\left(A\right)\right)$
is nonnegative. Hence 
\[
\left\Vert C_{w}\left(A\right)\right\Vert ^{2}_{2}\ge\sum_{v\in\mathrm{ch}\left(w\right)}\left\Vert C_{v}\left(A\right)\right\Vert ^{2}_{2}.
\]
Summing over all nodes $w$ with $\left|w\right|=n$ gives $M_{n}\left(A\right)\ge M_{n+1}\left(A\right)$,
and the displayed identity for the difference follows from the same
expansion. By the definition of $\Gamma_{n}\left(A\right)$, the monotonicity
of $\Gamma_{n}\left(A\right)$ is immediate. 
\end{proof}

We next consider the limit as $n\rightarrow\infty$. 
\begin{thm}
\label{thm:5-6} Let 
\[
\mathcal{B}_{n}:=\left\{ X\in\mathcal{S}_{2}:E_{n}\left(X\right)=X\right\} ,\qquad\mathcal{B}_{\infty}:=\bigcap_{n\ge0}\mathcal{B}_{n},
\]
and let $E_{\infty}$ be the orthogonal projection of $\mathcal{S}_{2}$
onto $\mathcal{B}_{\infty}$. Then for every $A\in\mathcal{S}_{2}\cap B\left(H\right)_{+}$
one has 
\[
\lim_{n\to\infty}\sum_{\left|w\right|=n}\left\Vert C_{w}\left(A\right)\right\Vert ^{2}_{2}=\left\Vert E_{\infty}\left(A\right)\right\Vert ^{2}_{2}.
\]
Equivalently, 
\[
\lim_{n\to\infty}\Gamma_{n}\left(A\right)=\frac{\left\Vert A\right\Vert ^{2}_{2}}{\left\Vert E_{\infty}\left(A\right)\right\Vert ^{2}_{2}},
\]
with the convention that the limit is $+\infty$ when $E_{\infty}\left(A\right)=0$.
\end{thm}

\begin{proof}
For each $n\ge0$, the map $E_{n}$ is the orthogonal projection of
$\mathcal{S}_{2}$ onto the closed subspace $\mathcal{B}_{n}$, by
the proof of \prettyref{lem:5-1}. Since every depth $n+1$ block
diagonal operator is also depth $n$ block diagonal, one has 
\[
\mathcal{B}_{n+1}\subset\mathcal{B}_{n}
\]
for every $n\ge0$. Thus $\left(\mathcal{B}_{n}\right)_{n\ge0}$ is
a decreasing sequence of closed subspaces of the Hilbert space $\mathcal{S}_{2}$.

Fix $X\in\mathcal{S}_{2}$. Since $\mathcal{B}_{m}\subset\mathcal{B}_{n}$
for $m\ge n$, the corresponding orthogonal projections satisfy 
\[
E_{m}E_{n}=E_{n}E_{m}=E_{m}.
\]
Hence, for $m\ge n$, 
\begin{align*}
\left\Vert E_{n}\left(X\right)-E_{m}\left(X\right)\right\Vert ^{2}_{2} & =\left\Vert E_{n}\left(X\right)\right\Vert ^{2}_{2}-2\Re\left\langle E_{n}\left(X\right),E_{m}\left(X\right)\right\rangle _{2}+\left\Vert E_{m}\left(X\right)\right\Vert ^{2}_{2}\\
 & =\left\Vert E_{n}\left(X\right)\right\Vert ^{2}_{2}-\left\Vert E_{m}\left(X\right)\right\Vert ^{2}_{2}.
\end{align*}
In particular, the sequence $\left\Vert E_{n}\left(X\right)\right\Vert ^{2}_{2}$
is nonincreasing and bounded below by $0$, so it converges, and the
displayed identity shows that $\left(E_{n}\left(X\right)\right)_{n\ge0}$
is Cauchy in $\mathcal{S}_{2}$. Let $Y$ be its limit.

For each fixed $j\ge0$, one has $E_{n}\left(X\right)\in\mathcal{B}_{j}$
whenever $n\ge j$. Since $\mathcal{B}_{j}$ is closed, it follows
that $Y\in\mathcal{B}_{j}$. Thus $Y\in\mathcal{B}_{\infty}$.

Now let $Z\in\mathcal{B}_{\infty}$. Then $Z\in\mathcal{B}_{n}$ for
every $n$, so 
\[
\left\langle X-E_{n}\left(X\right),Z\right\rangle _{2}=0
\]
for every $n$. Passing to the limit gives 
\[
\left\langle X-Y,Z\right\rangle _{2}=0.
\]
Hence $Y$ is the orthogonal projection of $X$ onto $\mathcal{B}_{\infty}$,
so $Y=E_{\infty}\left(X\right)$. Therefore 
\[
E_{n}\left(X\right)\to E_{\infty}\left(X\right)\quad\text{in }\mathcal{S}_{2}.
\]

Now apply this to $X=A$. By \prettyref{lem:5-1}, 
\[
\sum_{\left|w\right|=n}\left\Vert C_{w}\left(A\right)\right\Vert ^{2}_{2}=\mathrm{tr}\left(E_{n}\left(A\right)A\right).
\]
Since $E_{n}$ is the orthogonal projection onto $\mathcal{B}_{n}$,
one also has 
\[
\mathrm{tr}\left(E_{n}\left(A\right)A\right)=\left\langle E_{n}\left(A\right),A\right\rangle _{2}=\left\Vert E_{n}\left(A\right)\right\Vert ^{2}_{2}.
\]
Thus 
\[
\sum_{\left|w\right|=n}\left\Vert C_{w}\left(A\right)\right\Vert ^{2}_{2}=\left\Vert E_{n}\left(A\right)\right\Vert ^{2}_{2}.
\]
Passing to the limit gives 
\[
\lim_{n\to\infty}\sum_{\left|w\right|=n}\left\Vert C_{w}\left(A\right)\right\Vert ^{2}_{2}=\left\Vert E_{\infty}\left(A\right)\right\Vert ^{2}_{2}.
\]

Finally, whenever $\Gamma_{n}\left(A\right)$ is defined, 
\[
\Gamma_{n}\left(A\right)=\frac{\left\Vert A\right\Vert ^{2}_{2}}{\sum_{\left|w\right|=n}\left\Vert C_{w}\left(A\right)\right\Vert ^{2}_{2}}.
\]
If $E_{\infty}\left(A\right)\neq0$, the denominator converges to
$\left\Vert E_{\infty}\left(A\right)\right\Vert ^{2}_{2}>0$, so 
\[
\lim_{n\to\infty}\Gamma_{n}\left(A\right)=\frac{\left\Vert A\right\Vert ^{2}_{2}}{\left\Vert E_{\infty}\left(A\right)\right\Vert ^{2}_{2}}.
\]
If $E_{\infty}\left(A\right)=0$, then the denominator tends to $0$,
and therefore $\Gamma_{n}\left(A\right)\to+\infty$. 
\end{proof}

The space $\mathcal{B}_{\infty}$ consists of the HS operators that
are block diagonal at every depth, equivalently those commuting with
every packet projection $P_{w}$. Thus $E_{\infty}(A)$ is the component
of $A$ that remains block diagonal under arbitrarily fine packet
refinement. \prettyref{thm:5-6} identifies the limiting total squared
HS content with the squared HS norm of this component.

The following is the HS counterpart of \prettyref{prop:4-3}. 
\begin{prop}
\label{prop:5-6} Fix $n\ge1$. There exists $A\in\mathcal{S}_{2}\cap B\left(H\right)_{+}$
such that $A$ is block diagonal at depth $n$ and the depth $n$
HS greedy rule satisfies 
\[
\left\Vert A-D_{1}\right\Vert ^{2}_{2}=\left(1-\frac{1}{N_{n}}\right)\left\Vert A\right\Vert ^{2}_{2}.
\]
Consequently, the factor $1-\frac{1}{N_{n}}$ in \prettyref{eq:5-8}
is sharp among one-step bounds depending only on $N_{n}$ in the block
diagonal regime. 
\end{prop}

\begin{proof}
For each $w\in\mathcal{T}_{n}$ choose a unit vector $e_{w}\in W_{w}$,
and set 
\[
A=\sum_{w\in\mathcal{T}_{n}}\left|e_{w}\left\rangle \right\langle e_{w}\right|.
\]
Then $A\in\mathcal{S}_{2}\cap B\left(H\right)_{+}$ and 
\[
P_{w}A=AP_{w}
\]
for every $w\in\mathcal{T}_{n}$, so $A$ is block diagonal at depth
$n$. Moreover, 
\[
C_{w}\left(A\right)=A^{1/2}P_{w}A^{1/2}=P_{w}AP_{w}=\left|e_{w}\left\rangle \right\langle e_{w}\right|
\]
for every $w\in\mathcal{T}_{n}$. Hence 
\[
\left\Vert C_{w}\left(A\right)\right\Vert ^{2}_{2}=1
\]
for every $w\in\mathcal{T}_{n}$, and therefore 
\[
\left\Vert A\right\Vert ^{2}_{2}=\sum_{w\in\mathcal{T}_{n}}\left\Vert C_{w}\left(A\right)\right\Vert ^{2}_{2}=N_{n}.
\]
The HS greedy rule may choose any depth $n$ node, so 
\[
D_{1}=\left|e_{w_{1}}\left\rangle \right\langle e_{w_{1}}\right|
\]
for some $w_{1}\in\mathcal{T}_{n}$. Since the blocks are orthogonal
in $\mathcal{S}_{2}$, one has 
\begin{align*}
\left\Vert A-D_{1}\right\Vert ^{2}_{2} & =\sum_{w\in\mathcal{T}_{n}\setminus\left\{ w_{1}\right\} }\left\Vert C_{w}\left(A\right)\right\Vert ^{2}_{2}\\
 & =N_{n}-1=\left(1-\frac{1}{N_{n}}\right)N_{n}=\left(1-\frac{1}{N_{n}}\right)\left\Vert A\right\Vert ^{2}_{2}.
\end{align*}
\end{proof}

\section{Adaptive content extraction}\label{sec:6}

Sections \ref{sec:4} and \ref{sec:5} use packet contents at a fixed
depth. We now allow the resolution to vary along the tree through
finite active partitions up to a terminal depth $N$. Refining a node
replaces its content by the contents of its children. The resulting
change in total squared HS content is determined by the interaction
among those child contents, and these local quantities add along chains
of refinements.

\subsection{Active partitions and adaptive extraction}\label{sec:6-1}

Fix a terminal depth $N\ge1$ once and for all.
\begin{defn}
\label{def:6-1} A finite set $\Lambda\subset\mathcal{T}$ is called
an active partition up to depth $N$ if the following hold: 
\begin{enumerate}
\item each $w\in\Lambda$ satisfies $\left|w\right|\le N$, 
\item $\Lambda$ is an antichain, 
\item every node $z\in\mathcal{T}$ with $\left|z\right|=N$ has a unique
ancestor in $\Lambda$. 
\end{enumerate}
\end{defn}

Thus an active partition is a variable tree cut through the depth
$N$ packet decomposition. The extremal cases are $\left\{ \emptyset\right\} $
and $\mathcal{T}_{N}:=\left\{ w\in\mathcal{T}:\left|w\right|=N\right\} $. 
\begin{defn}
\label{def:6-2} Let $\Lambda$ be an active partition up to depth
$N$, and let $w\in\Lambda$ with $\left|w\right|<N$. The refinement
of $\Lambda$ at $w$ is 
\[
\Lambda\left[w\right]:=\left(\Lambda\setminus\left\{ w\right\} \right)\cup\mathrm{ch}\left(w\right).
\]
\end{defn}

\begin{lem}
\label{lem:6-3} Let $\Lambda$ be an active partition up to depth
$N$, and let $A\in B\left(H\right)_{+}$. Then 
\[
A=\sum_{w\in\Lambda}C_{w}\left(A\right).
\]
\end{lem}

\begin{proof}
For each $w\in\Lambda$, let $\mathcal{D}_{N}\left(w\right):=\left\{ z\in\mathcal{T}:\left|z\right|=N,\ z\succeq w\right\} $.
Since $\Lambda$ is an active partition, the sets $\mathcal{D}_{N}\left(w\right)$,
$w\in\Lambda$, form a partition of $\mathcal{T}_{N}$. By repeated
application of \prettyref{eq:2-3}, 
\[
C_{w}\left(A\right)=\sum_{z\in\mathcal{D}_{N}\left(w\right)}C_{z}\left(A\right)\qquad\left(w\in\Lambda\right).
\]
Summing over $w\in\Lambda$ gives 
\[
\sum_{w\in\Lambda}C_{w}\left(A\right)=\sum_{\left|z\right|=N}C_{z}\left(A\right)=A,
\]
using the fixed depth decomposition at level $N$. 
\end{proof}

For $A\in\mathcal{S}_{2}\cap B\left(H\right)_{+}$ and an active partition
$\Lambda$, define the total squared HS content at $\Lambda$ by 
\[
M_{\Lambda}\left(A\right):=\sum_{w\in\Lambda}\left\Vert C_{w}\left(A\right)\right\Vert ^{2}_{2}.
\]

If $w\in\Lambda$ with $\left|w\right|<N$, define the local refinement
defect at $w$ by 
\[
\delta_{w}\left(A\right):=\left\Vert C_{w}\left(A\right)\right\Vert ^{2}_{2}-\sum_{v\in\mathrm{ch}\left(w\right)}\left\Vert C_{v}\left(A\right)\right\Vert ^{2}_{2}.
\]

The next theorem identifies the precise cost of passing to finer resolution
at a single node. 
\begin{thm}
\label{thm:6-4} Let $\Lambda$ be an active partition up to depth
$N$, let $w\in\Lambda$ with $\left|w\right|<N$, and let $A\in\mathcal{S}_{2}\cap B\left(H\right)_{+}$.
Then 
\begin{equation}
M_{\Lambda}\left(A\right)-M_{\Lambda\left[w\right]}\left(A\right)=\delta_{w}\left(A\right).\label{eq:6-1}
\end{equation}
Moreover, 
\begin{equation}
\delta_{w}\left(A\right)=\sum_{\substack{v,v'\in\mathrm{ch}\left(w\right)\\
v\ne v'
}
}\mathrm{tr}\left(C_{v}\left(A\right)C_{v'}\left(A\right)\right)\ge0.\label{eq:6-2}
\end{equation}
In particular, 
\begin{equation}
M_{\Lambda\left[w\right]}\left(A\right)\le M_{\Lambda}\left(A\right).\label{eq:6-3}
\end{equation}
\end{thm}

\begin{proof}
Using the definition of $\Lambda\left[w\right]$, one has 
\[
M_{\Lambda}\left(A\right)-M_{\Lambda\left[w\right]}\left(A\right)=\left\Vert C_{w}\left(A\right)\right\Vert ^{2}_{2}-\sum_{v\in\mathrm{ch}\left(w\right)}\left\Vert C_{v}\left(A\right)\right\Vert ^{2}_{2}=\delta_{w}\left(A\right).
\]
This proves \prettyref{eq:6-1}.

By \prettyref{eq:2-3}, $C_{w}\left(A\right)=\sum_{v\in\mathrm{ch}\left(w\right)}C_{v}\left(A\right)$.
Hence 
\begin{align*}
\left\Vert C_{w}\left(A\right)\right\Vert ^{2}_{2} & =\mathrm{tr}\left(C_{w}\left(A\right)^{2}\right)\\
 & =\mathrm{tr}\left(\left(\sum_{v\in\mathrm{ch}\left(w\right)}C_{v}\left(A\right)\right)^{2}\right)\\
 & =\sum_{v\in\mathrm{ch}\left(w\right)}\left\Vert C_{v}\left(A\right)\right\Vert ^{2}_{2}+\sum_{\substack{v,v'\in\mathrm{ch}\left(w\right)\\
v\ne v'
}
}\mathrm{tr}\left(C_{v}\left(A\right)C_{v'}\left(A\right)\right).
\end{align*}
Then \prettyref{eq:6-2} holds, and \prettyref{eq:6-3} follows from
\prettyref{eq:6-1}. 
\end{proof}

\prettyref{thm:6-4} gives an additive refinement calculus along any
adaptive chain of active partitions.
\begin{cor}
\label{cor:6-5} Let $\Lambda_{0},\Lambda_{1},\dots,\Lambda_{m}$
be active partitions up to depth $N$, and suppose that for each $j=1,\dots,m$
there exists $w_{j}\in\Lambda_{j-1}$ with $\left|w_{j}\right|<N$
such that $\Lambda_{j}=\Lambda_{j-1}\left[w_{j}\right]$. Then for
every $A\in\mathcal{S}_{2}\cap B\left(H\right)_{+}$ one has 
\[
M_{\Lambda_{0}}\left(A\right)-M_{\Lambda_{m}}\left(A\right)=\sum^{m}_{j=1}\delta_{w_{j}}\left(A\right).
\]
Equivalently, 
\[
M_{\Lambda_{0}}\left(A\right)-M_{\Lambda_{m}}\left(A\right)=\sum^{m}_{j=1}\sum_{\substack{v,v'\in\mathrm{ch}\left(w_{j}\right)\\
v\ne v'
}
}\mathrm{tr}\left(C_{v}\left(A\right)C_{v'}\left(A\right)\right).
\]
\end{cor}

\begin{proof}
Apply \prettyref{thm:6-4} at each step and telescope: 
\[
M_{\Lambda_{0}}\left(A\right)-M_{\Lambda_{m}}\left(A\right)=\sum^{m}_{j=1}\left(M_{\Lambda_{j-1}}\left(A\right)-M_{\Lambda_{j}}\left(A\right)\right)=\sum^{m}_{j=1}\delta_{w_{j}}\left(A\right).
\]
The second formula follows from the explicit expression for $\delta_{w_{j}}\left(A\right)$
in \prettyref{thm:6-4}. 
\end{proof}

The positive residual process from \prettyref{sec:2} can be attached
to an adaptive chain of active partitions.
\begin{defn}
\label{def:6-6} Let $\Lambda_{0}$ be an active partition up to depth
$N$, and let $R\in B\left(H\right)_{+}$. An adaptive content extraction
process consists of sequences $\left(\Lambda_{k}\right)_{k\ge0}$,
$\left(w_{k}\right)_{k\ge1}$, $\left(D_{k}\right)_{k\ge1}$, $\left(R^{\left(k\right)}\right)_{k\ge0}$,
such that 
\begin{enumerate}
\item $R^{\left(0\right)}=R$, 
\item for each $k\ge1$, one has $w_{k}\in\Lambda_{k-1}$, 
\item if $\left|w_{k}\right|<N$, then $\Lambda_{k}=\Lambda_{k-1}\left[w_{k}\right]$,
and if $\left|w_{k}\right|=N$, one sets $\Lambda_{k}=\Lambda_{k-1}$, 
\item $D_{k}=C_{w_{k}}\left(R^{\left(k-1\right)}\right)$, $R^{\left(k\right)}=R^{\left(k-1\right)}-D_{k}$. 
\end{enumerate}
\end{defn}

\begin{cor}
\label{cor:6-7} Let $\left(\Lambda_{k},w_{k},D_{k},R^{\left(k\right)}\right)$
be an adaptive content extraction process. Then for every $m\ge1$
one has 
\[
R=R^{\left(m\right)}+\sum^{m}_{k=1}D_{k},\quad R^{\left(m\right)}\in B\left(H\right)_{+}.
\]
\end{cor}

\begin{proof}
The identity follows by iterating 
\[
R^{\left(k\right)}=R^{\left(k-1\right)}-D_{k},\quad D_{k}=C_{w_{k}}\left(R^{\left(k-1\right)}\right),
\]
exactly as in \prettyref{lem:2-2}. Since each $D_{k}=C_{w_{k}}\left(R^{\left(k-1\right)}\right)$
is positive and satisfies 
\[
0\le D_{k}\le R^{\left(k-1\right)},
\]
it follows that $R^{\left(k\right)}\in B\left(H\right)_{+}$ for every
$k$. 
\end{proof}

\subsection{Adaptive greedy extraction}\label{sec:6-2}

We next impose greedy selection on the adaptive process. 

Let $\Lambda$ be an active partition up to depth $N$. Since every
depth $N$ node has a unique ancestor in $\Lambda$, the subspaces
$\left\{ W_{w}:w\in\Lambda\right\} $ are pairwise orthogonal and
\[
H=\bigoplus_{w\in\Lambda}W_{w},\qquad I=\sum_{w\in\Lambda}P_{w}.
\]
Accordingly, for $X\in\mathcal{S}_{2}$ we define 
\[
E_{\Lambda}\left(X\right):=\sum_{w\in\Lambda}P_{w}XP_{w}.
\]
For $A\in\mathcal{S}_{2}\cap B\left(H\right)_{+}$, define 
\[
\Gamma_{\Lambda}\left(A\right):=\frac{\left\Vert A\right\Vert ^{2}_{2}}{\sum_{w\in\Lambda}\left\Vert C_{w}\left(A\right)\right\Vert ^{2}_{2}},
\]
whenever the denominator is nonzero.
\begin{lem}
\label{lem:6-8} Let $\Lambda$ be an active partition up to depth
$N$, and let $A\in\mathcal{S}_{2}\cap B\left(H\right)_{+}$. Then
\begin{equation}
\sum_{w\in\Lambda}\left\Vert C_{w}\left(A\right)\right\Vert ^{2}_{2}=\mathrm{tr}\left(E_{\Lambda}\left(A\right)A\right).\label{eq:6-4}
\end{equation}
Moreover, 
\begin{equation}
\frac{1}{\#\Lambda}\left\Vert A\right\Vert ^{2}_{2}\le\sum_{w\in\Lambda}\left\Vert C_{w}\left(A\right)\right\Vert ^{2}_{2}\le\left\Vert A\right\Vert ^{2}_{2},\label{eq:6-5}
\end{equation}
and hence 
\[
1\le\Gamma_{\Lambda}\left(A\right)\le\#\Lambda.
\]
Finally, $\Gamma_{\Lambda}\left(A\right)=1$ if and only if $E_{\Lambda}\left(A\right)=A$,
equivalently if and only if $P_{w}A=AP_{w}$ for all $w\in\Lambda$. 
\end{lem}

\begin{proof}
The proof is the same as that of \prettyref{lem:5-1}, with the fixed
depth family $\mathcal{T}_{n}$ replaced by the active partition $\Lambda$.
Since 
\[
H=\bigoplus_{w\in\Lambda}W_{w},\qquad I=\sum_{w\in\Lambda}P_{w},
\]
the map $E_{\Lambda}$ is the orthogonal projection of $\mathcal{S}_{2}$
onto the subspace of $\Lambda$ block diagonal operators. Also, \prettyref{lem:6-3}
gives 
\[
A=\sum_{w\in\Lambda}C_{w}\left(A\right).
\]
The identity \prettyref{eq:6-4}, the bounds \prettyref{eq:6-5},
and the characterization of the case $\Gamma_{\Lambda}\left(A\right)=1$
now follow from the same argument as in \prettyref{lem:5-1}. 
\end{proof}

Let $\left(\Lambda_{k},w_{k},D_{k},R^{\left(k\right)}\right)$ be
an adaptive content extraction process in the sense of \prettyref{def:6-6}.

We call the process trace greedy if 
\[
w_{k}\in\arg\max_{w\in\Lambda_{k-1}}\mathrm{tr}\left(C_{w}\left(R^{\left(k-1\right)}\right)\right)
\]
for every $k\ge1$.

We call the process HS greedy if 
\[
w_{k}\in\arg\max_{w\in\Lambda_{k-1}}\left\Vert C_{w}\left(R^{\left(k-1\right)}\right)\right\Vert _{2}
\]
for every $k\ge1$.

The next theorem gives the adaptive counterparts of the fixed depth
estimates from Sections \ref{sec:4} and \ref{sec:5}.
\begin{thm}
\label{thm:6-9} Let $\left(\Lambda_{k},w_{k},D_{k},R^{\left(k\right)}\right)$
be an adaptive content extraction process.
\begin{enumerate}
\item If the process is trace greedy, then for every $k\ge1$, 
\begin{equation}
\mathrm{tr}\left(R^{\left(k\right)}\right)\le\left(1-\frac{1}{\#\Lambda_{k-1}}\right)\mathrm{tr}\left(R^{\left(k-1\right)}\right).\label{eq:6-6}
\end{equation}
Consequently, for every $m\ge1$, 
\begin{equation}
\mathrm{tr}\left(R^{\left(m\right)}\right)\le\prod^{m-1}_{j=0}\left(1-\frac{1}{\#\Lambda_{j}}\right)\mathrm{tr}\left(R\right).\label{eq:6-7}
\end{equation}
\item If the process is HS greedy, then for every $k\ge1$, 
\begin{equation}
\left\Vert R^{\left(k\right)}\right\Vert ^{2}_{2}\le\left(1-\frac{1}{\Gamma_{\Lambda_{k-1}}\left(R^{\left(k-1\right)}\right)\#\Lambda_{k-1}}\right)\left\Vert R^{\left(k-1\right)}\right\Vert ^{2}_{2}.\label{eq:6-8}
\end{equation}
Consequently, for every $m\ge1$, 
\begin{equation}
\left\Vert R^{\left(m\right)}\right\Vert ^{2}_{2}\le\prod^{m-1}_{j=0}\left(1-\frac{1}{\Gamma_{\Lambda_{j}}\left(R^{\left(j\right)}\right)\#\Lambda_{j}}\right)\left\Vert R\right\Vert ^{2}_{2},\label{eq:6-9}
\end{equation}
and, using \prettyref{lem:6-8}, 
\begin{equation}
\left\Vert R^{\left(m\right)}\right\Vert ^{2}_{2}\le\prod^{m-1}_{j=0}\left(1-\frac{1}{\left(\#\Lambda_{j}\right)^{2}}\right)\left\Vert R\right\Vert ^{2}_{2}.\label{eq:6-10}
\end{equation}
\end{enumerate}
\end{thm}

\begin{proof}
For part (1), set $A=R^{\left(k-1\right)}$. By \prettyref{lem:6-3},
$A=\sum_{w\in\Lambda_{k-1}}C_{w}\left(A\right)$. Taking traces gives
\[
\mathrm{tr}\left(A\right)=\sum_{w\in\Lambda_{k-1}}\mathrm{tr}\left(C_{w}\left(A\right)\right).
\]
The right-hand side is a sum of $\#\Lambda_{k-1}$ nonnegative numbers.
Hence the maximal term is at least the average, so 
\[
\mathrm{tr}\left(D_{k}\right)\ge\frac{1}{\#\Lambda_{k-1}}\mathrm{tr}\left(R^{\left(k-1\right)}\right).
\]
Therefore 
\[
\mathrm{tr}\left(R^{\left(k\right)}\right)=\mathrm{tr}\left(R^{\left(k-1\right)}\right)-\mathrm{tr}\left(D_{k}\right)\le\left(1-\frac{1}{\#\Lambda_{k-1}}\right)\mathrm{tr}\left(R^{\left(k-1\right)}\right),
\]
which is \prettyref{eq:6-6}. Iterating gives \prettyref{eq:6-7}.

For part (2), again set $A=R^{\left(k-1\right)}$. Since $0\le D_{k}\le A$,
\prettyref{lem:5-2} gives 
\[
\left\Vert R^{\left(k\right)}\right\Vert ^{2}_{2}=\left\Vert A-D_{k}\right\Vert ^{2}_{2}\le\left\Vert A\right\Vert ^{2}_{2}-\left\Vert D_{k}\right\Vert ^{2}_{2}.
\]
By the HS greedy choice, 
\[
\left\Vert D_{k}\right\Vert ^{2}_{2}=\max_{w\in\Lambda_{k-1}}\left\Vert C_{w}\left(A\right)\right\Vert ^{2}_{2}\ge\frac{1}{\#\Lambda_{k-1}}\sum_{w\in\Lambda_{k-1}}\left\Vert C_{w}\left(A\right)\right\Vert ^{2}_{2}.
\]
Using the definition of $\Gamma_{\Lambda_{k-1}}\left(A\right)$, this
becomes 
\[
\left\Vert D_{k}\right\Vert ^{2}_{2}\ge\frac{1}{\Gamma_{\Lambda_{k-1}}\left(A\right)\#\Lambda_{k-1}}\left\Vert A\right\Vert ^{2}_{2}.
\]
Substituting into the previous inequality yields \prettyref{eq:6-8}.
Iterating gives \prettyref{eq:6-9}, and \prettyref{eq:6-10} follows
from 
\[
\Gamma_{\Lambda_{j}}\left(R^{\left(j\right)}\right)\le\#\Lambda_{j}
\]
from \prettyref{lem:6-8}. 
\end{proof}

The adaptive bounds above can be much smaller than the fixed depth
bounds when the active partitions remain far below the full packet
level. See the binary case below. 
\begin{prop}
\label{prop:6-10} Assume that each nonterminal node of $\mathcal{T}$
has exactly two children. Fix a terminal depth $N\ge2$, and start
from the active partition 
\[
\Lambda_{0}=\mathcal{T}_{1}.
\]
Suppose that for the first $m$ steps, one always refines a node of
depth strictly smaller than $N$. Then 
\[
\#\Lambda_{j}=j+2,\quad\text{for }0\le j\le m-1.
\]
Consequently, the adaptive trace bound \prettyref{eq:6-7} gives 
\begin{equation}
\mathrm{tr}\left(R^{\left(m\right)}\right)\le\frac{1}{m+1}\mathrm{tr}\left(R\right),\label{eq:6-11}
\end{equation}
and the uniform adaptive HS bound \prettyref{eq:6-10} gives 
\begin{equation}
\left\Vert R^{\left(m\right)}\right\Vert ^{2}_{2}\le\frac{m+2}{2m+2}\left\Vert R\right\Vert ^{2}_{2}.\label{eq:6-12}
\end{equation}
\end{prop}

\begin{proof}
Since each refinement replaces one node by two children, the cardinality
of the active partition increases by one at each step. Thus 
\[
\#\Lambda_{j}=j+2\qquad\text{for }0\le j\le m-1.
\]
Substituting this into \prettyref{eq:6-7} gives 
\[
\mathrm{tr}\left(R^{\left(m\right)}\right)\le\prod^{m-1}_{j=0}\left(1-\frac{1}{j+2}\right)\mathrm{tr}\left(R\right)=\prod^{m-1}_{j=0}\frac{j+1}{j+2}\mathrm{tr}\left(R\right)=\frac{1}{m+1}\mathrm{tr}\left(R\right),
\]
which is \prettyref{eq:6-11}.

Likewise, \prettyref{eq:6-10} gives 
\[
\left\Vert R^{\left(m\right)}\right\Vert ^{2}_{2}\le\prod^{m-1}_{j=0}\left(1-\frac{1}{\left(j+2\right)^{2}}\right)\left\Vert R\right\Vert ^{2}_{2}.
\]
Since 
\[
1-\frac{1}{\left(j+2\right)^{2}}=\frac{j+1}{j+2}\frac{j+3}{j+2},
\]
the product telescopes: 
\[
\prod^{m-1}_{j=0}\left(1-\frac{1}{\left(j+2\right)^{2}}\right)=\prod^{m-1}_{j=0}\frac{j+1}{j+2}\prod^{m-1}_{j=0}\frac{j+3}{j+2}=\frac{1}{m+1}\frac{m+2}{2}=\frac{m+2}{2m+2}.
\]
This proves \prettyref{eq:6-12}.
\end{proof}

\begin{example}
For the binary case, at terminal depth $N$, the fixed depth bounds
from Sections \ref{sec:4} and \ref{sec:5} are 
\begin{equation}
\mathrm{tr}\left(R^{\left(m\right)}\right)\le\left(1-\frac{1}{2^{N}}\right)^{m}\mathrm{tr}\left(R\right)\label{eq:6-13}
\end{equation}
and 
\begin{equation}
\left\Vert R^{\left(m\right)}\right\Vert ^{2}_{2}\le\left(1-\frac{1}{2^{2N}}\right)^{m}\left\Vert R\right\Vert ^{2}_{2}.\label{eq:6-14}
\end{equation}
These are the estimates from Theorems \ref{thm:4-1} and \ref{thm:5-3}
at depth $N$, with $N_{N}=2^{N}$. 

When $N=10$ and $m=50$, the adaptive factors in \prettyref{eq:6-11}
and \prettyref{eq:6-12} are 
\[
\frac{1}{51}\approx0.0196,\qquad\frac{52}{102}\approx0.5098,
\]
whereas the fixed depth factors in \prettyref{eq:6-13} and \prettyref{eq:6-14}
are 
\[
\left(1-\frac{1}{1024}\right)^{50}\approx0.9523,\qquad\left(1-\frac{1}{1024^{2}}\right)^{50}\approx0.99995.
\]
\end{example}

\section*{Declarations}
\begin{flushleft}
\textbf{Funding.} No funding was received for conducting this study.
\par\end{flushleft}

\begin{flushleft}
\textbf{Competing interests. }The authors have no competing interests
to declare that are relevant to the content of this article.
\par\end{flushleft}

\begin{flushleft}
\textbf{Data Availability. }No datasets were generated or analyzed
for this study, and all results are established analytically. 
\par\end{flushleft}

\bibliographystyle{amsalpha}
\bibliography{ref}

\end{document}